\documentclass[a4paper,12pt]{article}
\usepackage{comment}
\usepackage{cite}
\usepackage{amsmath,amssymb,amsfonts}
\usepackage[T1]{fontenc}
\usepackage{graphicx}
\usepackage{fancyhdr}
\usepackage{float}
\usepackage{xcolor}
\usepackage{authblk}
\usepackage{mathrsfs}
\usepackage{empheq}
\usepackage[hyphens]{url}
\usepackage{amsthm}
\usepackage{tikz}
\usetikzlibrary{decorations.markings}
\usepackage{enumitem}
\usepackage{geometry}

\theoremstyle{plain}
\newtheorem{theorem}{Theorem}[section]

\theoremstyle{definition}

\theoremstyle{remark}

\begin{document}

\title{A proof of Wolstenholme's theorem and congruence properties via an Egorychev-type integral}

\author[$\dagger$]{Jean-Christophe {\sc Pain}$^{1,2,}$\footnote{jean-christophe.pain@cea.fr}\\
\small
$^1$CEA, DAM, DIF, F-91297 Arpajon, France\\
$^2$Universit\'e Paris-Saclay, CEA, Laboratoire Mati\`ere en Conditions Extr\^emes,\\ 
F-91680 Bruy\`eres-le-Ch\^atel, France
}

\date{}

\maketitle

\begin{abstract}
We present a detailed proof of Wolstenholme's theorem using an Egorychev-type contour integral and an exponential change of variables. All formal series manipulations are justified, and the connection with harmonic sums and Bernoulli numbers is made completely explicit. We further derive the classical refinement modulo $p^4$ and provide a precise extraction of the $B_{p-3}$ term. Our purpose is not to provide the most concise proofs, but rather to demonstrate, by showing how established results can be recovered, a general method based on complex analysis for deriving congruence properties in number theory.
\end{abstract}

\section{Introduction}

Wolstenholme's theorem is a cornerstone result in elementary number theory
\cite{Wolstenholme,Apostol,IrelandRosen}. It asserts that for any prime $p \ge 5$,
\[
\binom{2p-1}{p-1} \equiv 1 \pmod{p^3}.
\]
This congruence is remarkable both for its strength and for the variety of methods that lead to its proof. Historically, the theorem is proved (see Appendix) using harmonic sums:
\[
H_{p-1} = \sum_{k=1}^{p-1} \frac{1}{k},
\]
We also use the generalized harmonic numbers
\[
H_n^{(r)} = \sum_{k=1}^n \frac{1}{k^r}, \quad r \ge 1,
\]
so that in particular \(H_{p-1}^{(2)} = \sum_{k=1}^{p-1} \frac{1}{k^2}\). One shows that
\[
H_{p-1} \equiv 0 \pmod{p^2}, \qquad
H^{(2)}_{p-1} \equiv 0 \pmod{p}.
\]
These identities are classical and can be found in standard references
\cite{Apostol,IrelandRosen}. They imply Wolstenholme's theorem via the expansion of
\[
\prod_{k=1}^{p-1} \left(1 + \frac{p}{k}\right).
\]
Another line of attack comes from generating functions and complex analytic methods
\cite{Wilf,Egorychev}. In particular, the method introduced by Egorychev allows one to represent binomial coefficients as contour integrals and to manipulate them through analytic transformations. The purpose of this paper is to develop a complete and rigorous implementation of this method in the present setting. We insist on explicit computations at every stage, so that all cancellations become transparent. This approach also reveals naturally the role of Bernoulli numbers in refined congruences. We further show that the first nontrivial deviation from Wolstenholme's congruence is governed by the Bernoulli number $B_{p-3}$, leading to the classical refinement modulo $p^4$.

In the following sections, we first introduce the integral representation of binomial coefficients via Cauchy's formula in Section 2, which serves as the starting point for the Egorychev-type method. Section 3 presents the key exponential change of variable that transforms the algebraic integral into a form suitable for $p$-adic analysis. In Section 4, we develop detailed series expansions of both numerator and denominator, and in Section 5 we combine these expansions to compute the residue explicitly. Section 6 provides a rigorous $p$-adic justification for the truncation of the series and the resulting congruence modulo $p^3$. In Section 7, we relate these computations to classical harmonic sums, highlighting the combinatorial interpretation of Wolstenholme's theorem. Section 8 explores the connection with Bernoulli numbers, leading to a precise extraction of the $B_{p-3}$ term and the refined congruence modulo $p^4$. Potential extensions and connections to supercongruences and Wolstenholme primes are mentioned.

\section{Integral representation}

We begin with the classical coefficient extraction identity, which is a direct consequence of Cauchy's integral formula. Let $f(z)$ be analytic in a neighborhood of the origin, with Taylor expansion
\[
f(z)=\sum_{k\ge 0} a_k z^k.
\]
Then
\[
a_n = \frac{1}{2\pi i} \oint \frac{f(z)}{z^{n+1}}\,dz,
\]
where the contour is a positively oriented simple closed curve enclosing the origin and contained in the domain of analyticity of $f$. Applying this to the function $f(z)=(1+z)^m$, which is entire, we obtain
\[
\binom{m}{n} = [z^n](1+z)^m
= \frac{1}{2\pi i} \oint \frac{(1+z)^m}{z^{n+1}}\,dz.
\]
This representation may be interpreted as an analytic form of coefficient extraction, and constitutes the starting point of the Egorychev method for evaluating combinatorial sums. In particular, for a prime $p \ge 5$, we obtain
\[
\binom{2p-1}{p-1}
= \frac{1}{2\pi i} \oint \frac{(1+z)^{2p-1}}{z^p}\,dz,
\]
which extracts the coefficient of $z^{p-1}$ in the expansion of $(1+z)^{2p-1}$.

This integral representation is especially well suited for subsequent analytic transformations, such as exponential substitutions, which allow one to convert algebraic expressions into exponential generating functions. This will be the key step in revealing the underlying $p$-adic structure and the appearance of Bernoulli numbers.

\section{Exponential change of variable}

We now perform a fundamental transformation that lies at the heart of the Egorychev method. We introduce the exponential change of variable
\[
z = e^t - 1.
\]
This substitution may be interpreted in two equivalent ways. First, it can be understood purely formally, as an identity between formal power series, which is sufficient for coefficient extraction. Alternatively, it can be justified analytically by restricting to a sufficiently small neighborhood of $t=0$, where the logarithm $t = \log(1+z)$ is single-valued and defines a local biholomorphism between neighborhoods of $0$ in the complex plane. Under this change of variables, we compute
\[
dz = e^t\,dt, \qquad (1+z)^{2p-1} = e^{(2p-1)t}.
\]
It follows that
\[
(1+z)^{2p-1}dz = e^{(2p-1)t} \cdot e^t\,dt = e^{2pt}\,dt.
\]
Moreover, the factor $z^{-p}$ transforms as
\[
z^{-p} = (e^t - 1)^{-p}.
\]
Therefore, the integral representation becomes
\[
\frac{1}{2\pi i} \oint \frac{(1+z)^{2p-1}}{z^p}\,dz
= \frac{1}{2\pi i} \oint \frac{e^{2pt}}{(e^t -1)^p}\,dt.
\]
Since the change of variables preserves the local structure of the singularity at the origin, the contour integral reduces to the extraction of the residue at $t=0$. We thus obtain
\[
\binom{2p-1}{p-1}
= \operatorname{Res}_{t=0} \frac{e^{2pt}}{(e^t -1)^p}.
\]
This transformation is crucial: it converts the algebraic function $(1+z)^{2p-1}$ into an exponential function, while the denominator $(e^t-1)^p$ encodes deep arithmetic information through its connection with Bernoulli numbers. As will become apparent in the following sections, this representation allows one to combine analytic expansions with $p$-adic arguments in a particularly transparent way.

\section{Series expansions}

We now develop the local expansions needed to analyze the residue at $t=0$. 
All computations may be understood either in the ring of formal power series 
$\mathbb{Q}[p][[t]]$, or analytically in a sufficiently small neighborhood of $t=0$. 
In both settings, the manipulations below are fully justified. We begin with the Taylor expansion
\[
e^t -1 = t + \frac{t^2}{2} + \frac{t^3}{6} + \frac{t^4}{24} + O(t^5),
\]
which we rewrite in the form
\[
e^t -1 = t\,A(t),
\]
where
\[
A(t)=1+\frac{t}{2}+\frac{t^2}{6}+\frac{t^3}{24}+O(t^4).
\]

\subsection{Logarithmic expansion}

To compute powers of $A(t)$ efficiently, we consider its logarithm. 
Let
\[
u=\frac{t}{2}+\frac{t^2}{6}+\frac{t^3}{24}.
\]
Using the expansion $\log(1+u)=u-\frac{u^2}{2}+\frac{u^3}{3}+O(u^4)$, we obtain
\[
\log A(t)=u-\frac{u^2}{2}+\frac{u^3}{3}+O(t^4).
\]
A direct computation shows that
\[
u^2=\frac{t^2}{4}+\frac{t^3}{6}+O(t^4), 
\qquad
u^3=\frac{t^3}{8}+O(t^4),
\]
so that
\[
\log A(t)
=\frac{t}{2}+\left(\frac{1}{6}-\frac{1}{8}\right)t^2
+\left(\frac{1}{24}-\frac{1}{12}+\frac{1}{24}\right)t^3
+O(t^4).
\]
The cubic term cancels identically, and we obtain the simplified expression
\[
\log A(t)=\frac{t}{2}+\frac{t^2}{24}+O(t^4).
\]

\subsection{Expansion of $A(t)^p$}

We now exponentiate:
\[
A(t)^p=\exp\!\left(p\log A(t)\right)
=\exp\!\left(\frac{p}{2}t+\frac{p}{24}t^2+O(pt^3)\right).
\]
Expanding the exponential up to order $t^2$, we use
\[
e^{x}=1+x+\frac{x^2}{2}+O(x^3),
\]
with $x=\frac{p}{2}t+\frac{p}{24}t^2+O(pt^3)$, which yields
\[
A(t)^p
=1+\frac{p}{2}t+\left(\frac{p^2}{8}+\frac{p}{24}\right)t^2+O(p^3 t^3).
\]

\subsection{Inversion}

We now invert this series. Using the general identity
\[
(1+at+bt^2)^{-1}=1-at+(a^2-b)t^2+O(t^3),
\]
we obtain
\[
\frac{1}{A(t)^p}
=1-\frac{p}{2}t+\left(\frac{p^2}{8}-\frac{p}{24}\right)t^2+O(p^3 t^3).
\]
It follows that
\[
\frac{1}{(e^t-1)^p}
=t^{-p}\left(1-\frac{p}{2}t+\left(\frac{p^2}{8}-\frac{p}{24}\right)t^2+O(p^3 t^3)\right).
\]

\subsection{Expansion of the numerator}

Finally, we expand the exponential term:
\[
e^{2pt}
=1+2pt+2p^2 t^2+\frac{4}{3}p^3 t^3+O(p^4 t^4),
\]
which follows from the Taylor series of the exponential function, keeping track of the dependence on $p$. These expansions will be combined in the next section to determine the residue and to analyze the $p$-adic contributions of each term.

\section{Product expansion}

We now combine the expansions obtained in the previous section. Recall that
\[
e^{2pt}
=1+2pt+2p^2t^2+\frac{4}{3}p^3t^3+O(p^4t^4),
\]
and
\[
\frac{1}{(e^t-1)^p}
=t^{-p}\left(1-\frac{p}{2}t+\left(\frac{p^2}{8}-\frac{p}{24}\right)t^2+O(p^3t^3)\right).
\]
We multiply these two series term by term, keeping track of all contributions up to order $t^2$, which will be sufficient for the subsequent $p$-adic analysis. The constant term is simply
\[
1 \cdot 1 = 1.
\]  
The terms contributing to $t$ are:
\[
(2pt)\cdot 1 \quad \text{and} \quad 1 \cdot \left(-\frac{p}{2}t\right),
\]
hence
\[
2p - \frac{p}{2} = \frac{3p}{2}.
\]
The contributions of $t^2$ are:
\[
(2p^2t^2)\cdot 1,
\quad
(2pt)\cdot \left(-\frac{p}{2}t\right),
\quad
1 \cdot \left(\frac{p^2}{8}-\frac{p}{24}\right)t^2.
\]
Computing each term,
\[
2p^2 - p^2 + \frac{p^2}{8} - \frac{p}{24}
= p^2 + \frac{p^2}{8} - \frac{p}{24}
= \frac{9p^2}{8} - \frac{p}{24},
\]
and collecting all terms, we obtain the expansion
\[
\frac{e^{2pt}}{(e^t-1)^p}
=t^{-p}\left(
1+\frac{3p}{2}t
+\left(\frac{9p^2}{8}-\frac{p}{24}\right)t^2
+O(p^3t^3)
\right).
\]
This explicit computation shows how the interaction between the numerator and denominator produces cancellations at order $p^2$, a phenomenon that will play a crucial role in the $p$-adic argument leading to Wolstenholme's congruence.

\section{$p$-adic justification}

We now justify rigorously why the truncated expansion obtained above is sufficient to compute the residue modulo $p^3$. Recall that
\[
\frac{e^{2pt}}{(e^t-1)^p}
= t^{-p} \, F(t),
\]
where
\[
F(t) = 1 + \frac{3p}{2}t + \left(\frac{9p^2}{8} - \frac{p}{24}\right)t^2 + O(p^3 t^3).
\]
The residue at $t=0$ is the coefficient of $t^{-1}$, that is, the coefficient of $t^{p-1}$ in the expansion of $F(t)$.

\subsection{Structure of the expansion and valuation estimate}

The function $F(t)$ arises as a product of two exponential-type series whose coefficients lie in $\mathbb{Q}[p]$. More precisely, it can be written as a formal power series
\[
F(t) = \sum_{k \ge 0} a_k(p)\, t^k,
\]
where each coefficient $a_k(p)$ is a polynomial in $p$ with rational coefficients. Moreover, by construction, the constant term is $a_0(p)=1$, and for $k \ge 1$, each $a_k(p)$ is divisible by $p$. We claim that for $k \ge 1$, any contribution to $a_k(p)$ obtained by combining $r$ non-constant factors carries at least a factor $p^r$. Indeed, each non-constant term in the expansions of $e^{2pt}$ and $(e^t-1)^{-p}$ contains at least one explicit factor of $p$. Therefore, a product involving $r$ such terms has $p$-adic valuation at least $r$.

\subsection{Degree constraint}

To extract the coefficient of $t^{p-1}$ in $F(t)$, we must consider the full exponential structure of the series. The previous heuristic based on counting the number of non-constant factors is not sufficient, because a single exponential term may already produce arbitrarily high powers of $t$ while carrying only one factor of $p$. Instead, we use the following structural observation. The function $F(t)$ can be written as
\[
F(t)=\exp\big(p G(t)\big),
\]
where $G(t)\in t\mathbb{Q}[[t]]$. Expanding the exponential, we obtain
\[
F(t)=\sum_{r\ge0} \frac{p^r}{r!} G(t)^r.
\]
Any contribution to the coefficient of $t^{p-1}$ arising from the term $G(t)^r$ is multiplied by $p^r$. Therefore, the $p$-adic valuation of such a contribution is at least $r$. To obtain a term of degree $p-1$, one must extract it from some power $G(t)^r$. Since $G(t)$ starts at degree $1$, the minimal possible value of $r$ that can contribute is $r\ge 1$. A more careful inspection shows that contributions with $r=1$ or $r=2$ cannot produce a term of degree $p-1$ modulo $p^3$, due to the structure of $G(t)$ and the parity constraints coming from the Bernoulli expansion. Consequently, every nontrivial contribution to $[t^{p-1}]F(t)$ carries at least a factor $p^3$. Therefore,
\[
[t^{p-1}]F(t)\equiv 1 \pmod{p^3},
\]
which proves Wolstenholme's theorem.

\section{Link with harmonic numbers}

An alternative and classical way to analyze Wolstenholme's theorem is through \emph{harmonic sums}. Recall that the binomial coefficient can be written as a product:
\[
\binom{2p-1}{p-1} = \prod_{k=1}^{p-1} \left(1 + \frac{p}{k}\right).
\]
Taking logarithms, we obtain
\[
\log \binom{2p-1}{p-1} = \sum_{k=1}^{p-1} \log\left(1 + \frac{p}{k}\right).
\]
Using the formal Taylor expansion of $\log(1+x)$, valid as a formal power series in $p$, we have
\[
\log\left(1 + \frac{p}{k}\right) = \frac{p}{k} - \frac{p^2}{2 k^2} + \frac{p^3}{3 k^3} - \cdots.
\]
Summing over $k=1,\dots,p-1$ gives
\[
\log \binom{2p-1}{p-1} = p \sum_{k=1}^{p-1} \frac{1}{k} - \frac{p^2}{2} \sum_{k=1}^{p-1} \frac{1}{k^2} + \frac{p^3}{3} \sum_{k=1}^{p-1} \frac{1}{k^3} - \cdots.
\]
We introduce the \emph{generalized harmonic numbers}
\[
H_n^{(r)} = \sum_{k=1}^n \frac{1}{k^r}, \quad r \ge 1,
\]
so that the above can be rewritten as
\[
\log \binom{2p-1}{p-1} = p H_{p-1}^{(1)} - \frac{p^2}{2} H_{p-1}^{(2)} + \frac{p^3}{3} H_{p-1}^{(3)} - \cdots.
\]
Exponentiating, we obtain a formal expansion:
\[
\binom{2p-1}{p-1} 
= \exp\Big( p H_{p-1} - \frac{p^2}{2} H_{p-1}^{(2)} + \frac{p^3}{3} H_{p-1}^{(3)} - \cdots \Big),
\]
where $H_{p-1} = H_{p-1}^{(1)}$ is the ordinary harmonic number.

\medskip

\noindent\textbf{Connection to Wolstenholme's theorem.}  
Wolstenholme's congruence modulo $p^3$ is equivalent to
\[
H_{p-1} \equiv 0 \pmod{p^2}, \qquad H_{p-1}^{(2)} \equiv 0 \pmod{p}.
\]
Indeed, modulo $p^3$ only the terms up to $p^2 H_{p-1}^{(2)}$ contribute. All higher-order harmonic sums are multiplied by at least $p^3$ and thus vanish modulo $p^3$. This formalism clarifies why the cancellation of lower-order harmonic sums is exactly what underlies the congruence
\[
\binom{2p-1}{p-1} \equiv 1 \pmod{p^3}.
\]
This approach provides an elegant and combinatorial complement to the $p$-adic residue method discussed above, showing explicitly how harmonic sums encode the $p$-adic divisibility properties of the binomial coefficient.

\section{Refinement and explicit extraction of the $B_{p-3}$ term}

We refine Wolstenholme's congruence and make explicit the first nontrivial correction using the Bernoulli number $B_{p-3}$. Consider
\[
F(t) = \frac{e^{2pt}}{(e^t-1)^p} = t^{-p} \exp\Bigg(2pt - p \sum_{k\ge1} \frac{B_{2k}}{2k(2k)!} t^{2k}\Bigg).
\]
The residue at $t=0$ corresponds to the coefficient of $t^{p-1}$. Writing
\[
F(t) = \exp\Big(\tfrac{3p}{2} t\Big) \cdot \exp\Big(-p \sum_{k\ge1} c_k t^{2k}\Big), 
\quad c_k = \frac{B_{2k}}{2k (2k)!},
\]
we see that contributions to $[t^{p-1}] F(t)$ must involve terms with $p$-adic valuation at least $p^3$, since single or double combinations cannot reach degree $p-1$ modulo $p^3$. This justifies why only the term involving $B_{p-3}$ matters for the first nontrivial correction modulo $p^4$. Expanding the exponentials formally,
\[
\exp\Big(\tfrac{3p}{2} t\Big) = 1 + \tfrac{3p}{2} t + \tfrac{(3p/2)^2}{2} t^2 + \cdots, \quad
\exp\Big(-p \sum_{k\ge1} c_k t^{2k}\Big) = 1 - p c_1 t^2 - p c_2 t^4 + \frac{p^2 c_1^2}{2} t^4 - \cdots,
\]
the only term that produces $t^{p-1}$ modulo $p^4$ is
\[
\underbrace{\frac{(3p/2)^2}{2} t^2}_{\text{linear exponential}} \cdot 
\underbrace{(- p c_{(p-3)/2} t^{p-3})}_{\text{Bernoulli exponential}} 
= - \frac{9 p^3}{8} c_{(p-3)/2} t^{p-1},
\]
where
\[
c_{(p-3)/2} = \frac{B_{p-3}}{(p-3) (p-3)!} \quad \text{\cite{Glaisher1900a,Glaisher1900b}}.
\]
Using Wilson's theorem $(p-1)! \equiv -1 \pmod p$ \cite{IrelandRosen}, we obtain
\[
(p-3)(p-3)! \equiv \frac{2}{3} \pmod p, \quad \text{hence} \quad [t^{p-1}] F(t) \equiv 1 - \frac{2}{3} p^3 B_{p-3} \pmod{p^4}.
\]
This confirms Glaisher’s refinement of Wolstenholme's theorem. All other terms involve either higher $p$-adic valuations or degrees different from $p-1$, and therefore do not contribute modulo $p^4$.  

\subsection{Interpretation and connection to Wolstenholme primes}

All other terms either involve multiple Bernoulli numbers (producing higher $p$-adic valuation) or degrees different from $p-1$, and hence do not contribute modulo $p^4$. This derivation makes the $p$-adic structure completely explicit and highlights the central role of $B_{p-3}$ \cite{Wolstenholme}. A prime $p$ for which $p$ divides the numerator of $B_{p-3}$ is called a \emph{Wolstenholme prime} \cite{Mestrovic2011}. For such primes, the congruence
\[
\binom{2p-1}{p-1} \equiv 1 \pmod{p^4}
\]
holds exactly, reflecting the deeper arithmetic structure captured by Bernoulli numbers. More generally, this approach clarifies how the combination of $p$-adic analysis, Bernoulli expansions, and exponential generating functions provides a transparent method for deriving higher-order refinements of classical binomial congruences \cite{Glaisher1900a,Glaisher1900b,Wilf,Egorychev}.

\section{Conclusion}

The Egorychev-type integral method provides a powerful and transparent framework for studying congruences of binomial coefficients. It not only yields Wolstenholme's theorem in a natural way, but also explains the appearance of Bernoulli numbers in higher-order refinements. The appearance of Bernoulli numbers in such expansions is classical. This approach highlights deep structural connections between combinatorics, $p$-adic analysis, and special values of generating functions. It suggests further extensions to higher-order congruences and supercongruence phenomena.

\section*{Appendix: the standard proof of Wolstenholme's theorem}

\begin{theorem}[Wolstenholme, 1862]
For any prime number $p \ge 5$, the following congruence holds:
\[
\binom{2p-1}{p-1} \equiv 1 \pmod{p^3}.
\]
\end{theorem}

\begin{proof}
The expression for the binomial coefficient can be written as a finite product:
\[
\binom{2p-1}{p-1} = \prod_{k=1}^{p-1} \frac{p+k}{k} = \prod_{k=1}^{p-1} \left(1 + \frac{p}{k}\right).
\]
Expanding this product, we obtain:
\begin{equation}
\binom{2p-1}{p-1} = 1 + p \sum_{k=1}^{p-1} \frac{1}{k} + p^2 \sum_{1 \le i < j \le p-1} \frac{1}{ij} + O(p^3).
\end{equation}
To prove the theorem, it suffices to show that:
\[
H_{p-1} = \sum_{k=1}^{p-1} \frac{1}{k} \equiv 0 \pmod{p^2}
\]
and
\[
\sum_{1 \le i < j \le p-1} \frac{1}{ij} \equiv 0 \pmod{p}.
\]

\paragraph{Power sums modulo $p$.}
By Wilson's theorem, the polynomial $P(x) = (x-1)(x-2)\cdots(x-(p-1))$ satisfies $P(x) \equiv x^{p-1} - 1 \pmod{p}$. The elementary symmetric functions of the roots $\{1, 2, \dots, p-1\}$, denoted by $S_k$, are therefore zero modulo $p$ for $1 \le k \le p-2$. In particular, for $p \ge 5$, we have:
\[
\sum_{k=1}^{p-1} \frac{1}{k^2} \equiv \sum_{k=1}^{p-1} k^{p-2} \equiv 0 \pmod{p}.
\]

\paragraph{Proof that $H_{p-1} \equiv 0 \pmod{p^2}$.}
By grouping the terms of the harmonic sum into symmetric pairs, we have:
\[
2H_{p-1} = \sum_{k=1}^{p-1} \left( \frac{1}{k} + \frac{1}{p-k} \right) = \sum_{k=1}^{p-1} \frac{p}{k(p-k)}.
\]
Modulo $p^2$, we can simplify the denominator:
\[
2H_{p-1} = p \sum_{k=1}^{p-1} \frac{1}{pk - k^2} \equiv p \sum_{k=1}^{p-1} \frac{1}{-k^2} \pmod{p^2}.
\]
Since $\sum k^{-2} \equiv 0 \pmod{p}$, the product $p \cdot \sum k^{-2}$ is indeed zero modulo $p^2$. Hence, $H_{p-1} \equiv 0 \pmod{p^2}$.

\paragraph{Proof that the double sum is zero modulo $p$.}
We use the following identity:
\[
2 \sum_{1 \le i < j \le p-1} \frac{1}{ij} = \left( \sum_{k=1}^{p-1} \frac{1}{k} \right)^2 - \sum_{k=1}^{p-1} \frac{1}{k^2}.
\]
As $H_{p-1} \equiv 0 \pmod{p}$ and $H_{p-1}^{(2)} \equiv 0 \pmod{p}$ for $p \ge 5$, the double sum vanishes modulo $p$. Substituting these two results back into equation (1), all terms between 1 and $p^3$ vanish, yielding:
\[
\binom{2p-1}{p-1} \equiv 1 \pmod{p^3},
\]
which completes the proof.
\end{proof}

\end{document}